\documentclass[11pt,oneside,english,jou]{amsart}
\usepackage[T1]{fontenc}
\usepackage[latin9]{inputenc}
\usepackage{color}
\usepackage{verbatim}
\usepackage{amsthm}
\usepackage{amssymb}

\makeatletter
\numberwithin{equation}{section}
\numberwithin{figure}{section}
\theoremstyle{plain}
\newtheorem{thm}{Theorem}

\@ifundefined{definecolor}
 {\usepackage{color}}{}

\newcommand{\K}{\mathcal{K}}

\makeatother

\usepackage{babel}

\makeatother

\usepackage{babel}

\makeatother

\usepackage{babel}

\begin{document}

\title{A New Short Proof for the Uniqueness of the Universal Minimal Space}


\author{Yonatan Gutman}
\author{Hanfeng Li}
\address{\hskip-\parindent
Yonatan Gutman, Laboratoire d'Analyse et de Mathématiques Appliquées,
Université de Marne-la-Vallée, 5 Boulevard Descartes, Cité Descartes
- Champs-sur-Marne, 77454 Marne-la-Vallée cedex 2, France.}
\email{yonatan.gutman@univ-mlv.fr}

\address{\hskip-\parindent
Hanfeng Li, Department of Mathematics, SUNY at Buffalo,
Buffalo NY 14260-2900, U.S.A.}
\email{hfli@math.buffalo.edu}


\date{June 11, 2011}
\begin{abstract}
We give a new short proof for the uniqueness of the universal minimal
space. The proof holds for the uniqueness of the universal object in every collection of topological dynamical systems closed
under taking projective limits and possessing universal objects.
\end{abstract}

\maketitle
\section{Introduction}

%
{}

In topological dynamics one studies jointly continuous actions
$G\times X\rightarrow X$ of (Hausdorff) topological groups $G$ on nonempty
(Hausdorff) compact spaces $X$. Such an action is called a {\it topological dynamical system}, and we call
$X$ a {\it $G$-space}.
A $G$-space $X$ is said to be \textit{minimal}
if $X$ and $\emptyset$ are the only $G$-invariant closed subsets
of $X$. By Zorn's lemma each $G$-space contains a minimal $G$-subspace.
These minimal objects are in some sense the most basic ones in the
category of $G$-spaces. For various topological groups $G$ they
have been the object of extensive study. Given a topological group $G$,
one is naturally interested in describing all of the minimal $G$-spaces up to isomorphism.
Such a description is given by the following construction: one can
show that there exists a minimal $G$-space $U_G$
with the universal property that every minimal $G$-space $X$ is a {\it factor} of
$U_G$, i.e., there is a continuous $G$-equivariant map from
$U_{G}$ onto $X$.
Any such $G$-space is called a \textit{universal minimal
$G$-space}, however it can be shown to be unique up to isomorphism. The existence of a universal minimal $G$-space is easy
to demonstrate by choosing a minimal $G$-subspace of the product over all
minimal $G$-spaces (one representative from each isomorphism class
- the collection of isomorphism classes of minimal $G$-spaces is a set).
The uniqueness turns
out harder to show, since for two universal minimal $G$-spaces $X$ and $Y$, there could be more than one epimorphism from $X$ to $Y$,
where an {\it epimorphism} is a surjective $G$-equivariant continuous map.
An easy observation is that it suffices to show that
a universal minimal $G$-space $X$ is \textit{coalescent}, i.e.
every epimorphism $\phi:X\rightarrow X$ is an isomorphism.
If $M_1$
and $M_2$ are universal minimal $G$-spaces then by universality
we have epimorphisms $\phi_1: M_1\rightarrow M_2$
and $\phi_2: M_2\rightarrow M_1$. If in addition $M_1$
is coalescent, then $\phi_2\circ\phi_1$ must be an isomorphism, and hence
$\phi_1$ and $\phi_2$ are  isomorphisms.

In the literature two different approaches for the proof of uniqueness
have been offered. The algebraic approach using Ellis semigroups was laid out by Ellis in
\cite{Ellis69} (he only treated the case of discrete $G$ but the proof works for any $G$).
The details can be found in \cite{de Vries} IV(4.30)
and \cite{Uspenskij} Appendix 3.
The main tool is the theorem that any minimal subsystem of an enveloping
semigroup (see \cite{de Vries} IV(3.5(1)) is coalescent. This theorem
is proven by using the Ellis-Numakura Lemma (\cite{Ellis58}, \cite{Numakura}):
every non-empty compact
Hausdorff right semitopological semigroup contains an idempotent.
Another proof was given by Auslander in \cite{Auslander}. The main
idea is to show that given a minimal system $(G,X)$ there is a cardinal
$\kappa$ such that $(G,X^{\kappa})$ contains a minimal coalescent
subsystem. The main tool is the so called \textit{almost periodic sets}.

We will give a proof for the uniqueness of the universal minimal space that holds in a more general setting. Given a collection $\K$ of $G$-spaces,
we call $M\in\K$, \textit{\textcolor{black}{$\K$-universal }}if
any $N\in\K$ is a $G$-factor of $M$. We show that if $\K$ is closed
under taking projective limits and posses universal spaces then it
has a unique universal space (up to isomorphism). In particular the reader
will be able to verify easily, that our result also applies to the universal
equicontinuous minimal space and the universal distal minimal space.

%
{}

\medskip{}

\noindent \textit{Acknowledgements.} H. Li was partially supported
by NSF grants DMS-0701414 and DMS-1001625. We are grateful to Eli
Glasner for helpful discussions. We thank the referee for useful comments.

\section{Uniqueness}
\begin{thm}
Let $G$ be a topological group and $\K$ a collection of $G$-spaces.
Assume that $\K$ is closed under taking projective limits. Let $M_{1},M_{2}$
be $\K$-universal spaces. Then $M_1$ and $M_2$ are isomorphic.
\end{thm}
\begin{proof}
As noted in the introduction it suffices to prove that every $\K$-universal
space $M$ is coalescent. Let $\varphi:M\rightarrow M$ be an epimorphism. We will show that $\varphi$ is injective. Assume
not. Let $\beta$ be an ordinal with $|\beta|>|M^{2}|$. We will define
for any ordinal $\alpha\leq\beta$ a $G$-space $X_{\alpha}\in\K$
and epimorphisms $\phi_{\gamma, \alpha}:X_{\alpha}\rightarrow X_{\gamma}$
for any ordinal $\gamma\leq\alpha$ with the following properties:

\begin{enumerate}
\item \label{enu:(Compatibility)}(Compatibility) $\phi_{\delta, \alpha}=\phi_{\delta, \gamma}\circ\phi_{\gamma, \alpha}$
for $\alpha\geq\gamma\geq\delta$.
\item \label{enu:(Cardinality)} For each ordinal $\gamma<\alpha$, there
exist distinct $x_{\gamma},y_{\gamma}\in X_{\gamma+1}$ with $\phi_{\gamma,\gamma+1}(x_{\gamma})=\phi_{\gamma,\gamma+1}(y_{\gamma})$.
We denote this common image by $z_{\gamma}$.
\end{enumerate}

By the universality of $(G,M)$ one has an epimorphism $M\rightarrow X_{\beta}$.
This implies $|X_{\beta}|\leq|M|$. For each $\gamma<\beta$, since
$\phi_{\gamma+1, \beta}$ is surjective, we can find $\tilde{x}_{\gamma},\tilde{y}_{\gamma}\in X_{\beta}$
with $\phi_{\gamma+1, \beta}(\tilde{x}_{\gamma})=x_{\gamma}$ and $\phi_{\gamma+1, \beta}(\tilde{y}_{\gamma})=y_{\gamma}$.
For any $\gamma<\alpha<\beta$, one has $\phi_{\gamma+1, \beta}(\tilde{x}_{\alpha})=\phi_{\gamma+1, \alpha}(z_{\alpha})=\phi_{\gamma+1, \beta}(\tilde{y}_{\alpha})$,
and $\phi_{\gamma+1, \beta}(\tilde{x}_{\gamma})=x_{\gamma}\neq y_{\gamma}=\phi_{\gamma+1, \beta}(\tilde{y}_{\gamma})$,
and hence $(\tilde{x}_{\alpha},\tilde{y}_{\alpha})\neq(\tilde{x}_{\gamma},\tilde{y}_{\gamma})$.
This implies that the map $\{\gamma|\,0\leq\gamma<\beta\}\rightarrow X_{\beta}\times X_{\beta}$
given by $\gamma\mapsto(\tilde{x}_{\gamma},\tilde{y}_{\gamma})$ is
injective, which in turn implies that $|\beta|\leq|X_{\beta}^{2}|$.
Putting all the inequalities together including our initial choice
$|\beta|>|M^{2}|$, we get $|M^{2}|<|\beta|\leq|X_{\beta}^{2}|\leq|M^{2}|$,
which is impossible.

The construction is carried out through transfinite induction. Let $X_{0}=M$.
If $\alpha$ is a successor ordinal, take an epimorphism $f_{\alpha}:M\rightarrow X_{\alpha-1}$ using the universality
of $M$ and define $(G,X_{\alpha})=(G,M)$, $\phi_{\alpha-1, \alpha}=f_{\alpha}\circ\varphi$,
and for $\gamma<\alpha-1$, define $\phi_{\gamma, \alpha}=\phi_{\gamma, \alpha-1}\circ\phi_{\alpha-1, \alpha}$.
If $\alpha$ is a limit ordinal,
define $X_{\alpha}\in\K$  to be the projective
limit of $(X_{\gamma})_{\gamma<\alpha}$, and for
$\gamma<\alpha$, define $\phi_{\gamma, \alpha}:X_{\alpha}\rightarrow X_{\gamma}$ to be the epimorphism
coming from the projective limit. Clearly $X_{\alpha}$
is a $G$-space, and the conditions (\ref{enu:(Compatibility)})
and (\ref{enu:(Cardinality)}) are satisfied.
\end{proof}

\Small

\end{document}